\documentclass[12pt,reqno]{amsart}
\usepackage[utf8]{inputenc}
\usepackage{amsmath,amsthm,amssymb}
\usepackage{mathtext}
\usepackage[T1,T2A]{fontenc}
\usepackage[english]{babel}
\usepackage{tikz}

\newcommand{\KK}{\mathbb K}
\newcommand{\QQ}{\mathbb Q}
\newcommand{\ZZ}{\mathbb Z}
\newcommand{\Ga}{\mathbb G_a}
\newcommand{\Aut}{\mathrm{Aut}}
\newcommand{\Ker}{\mathrm{Ker}\,}
\newcommand{\SAut}{\mathrm{SAut}}
\newcommand{\Xreg}{X^{\mathrm{reg}}}
\newcommand{\LND}{\mathrm{LND}}
\renewcommand{\phi}{\varphi}

\usepackage{graphicx}
\usepackage[dvipsnames]{xcolor}
\usepackage[colorlinks=true, allcolors=red]{hyperref}
\usepackage{cite}

\theoremstyle{plain}
\newtheorem{theorem}{Theorem}[section]
\newtheorem{lemma}[theorem]{Lemma}
\newtheorem{cor}[theorem]{Corollary}

\theoremstyle{remark}

\theoremstyle{definition}
\newtheorem{definition}[theorem]{Definition}
\newtheorem{example}[theorem]{Example}

\title{}

\author{Sergey Gaifullin}
\address{Lomonosov Moscow State University, Faculty of Mechanics and Mathematics, Department of Higher Algebra, Leninskie Gory 1, Moscow, 119991 Russia;
\linebreak
Moscow Center of Fundamental and Applied Mathematics, Moscow, Russia;  
\linebreak
HSE University, Faculty of Computer Science, Pokrovsky Boulvard 11, Moscow, 109028 Russia}
\email{sgayf@yandex.ru}
\author{Veronika Kikteva}
\address{HSE University, Faculty of Computer Science, 11 Pokrovsky Bulvar, Moscow, 109028, Russia}
\email{VVKikteva@yandex.ru}

\thanks{The article was prepared within the framework of the project “International Academic Cooperation”
HSE University.}



\begin{document}
\title{Flexibility criterion for affine horospherical varieties}
\maketitle
\begin{abstract}
In this paper we obtain a criterion of flexibility for an affine complexity-zero horospherical variety. This result generalizes previously known results on flexibility of normal horospherical varieties, horospherical varieties with an action of a semisimple group, and non-normal toric varieties. 
\end{abstract}

\section{Introduction}
Let $\KK$ be an algebraically closed field of characteristic zero. For an affine variety $X$ over~$\KK$ its group of regular automorphisms $\Aut(X)$ is, in general, not an algebraic group. This leads to some interesting phenomena related to automorphism groups that do not appear in theory of algebraic groups. One such phenomenon is that $\Aut(X)$ can act {\it infinitely transitively} on an orbit, that is, for each $m\in\ZZ_{>0}$ it acts transitively on the set of $m$-tuples of distinct points of the orbit. 

An algebraic subgroup of $\Aut(X)$ is called a $\Ga$-subgroup if it is isomorphic to the additive group of the ground field $\KK$. We denote by $\SAut(X)$ the subgroup of $\Aut(X)$ generated by all $\Ga$-subgroups and call it the {\it group of special automorphisms}. An affine algebraic variety~$X$ is called \textit{flexible} if every regular point $x\in\Xreg$ is flexible, that is, its tangent space $\mathrm{T}_x X$ is spanned by tangent vectors to $\Ga$-orbits. According to~\cite[Theorem~0.1]{AFKKZ-1}, for an affine variety $X$ of dimension $\geq 2$, flexibility of $X$ is equivalent to transitivity of $\SAut(X)$-action on the regular locus $\Xreg$ and is equivalent to infinite transitivity of $\SAut(X)$-action on~$\Xreg$. This fact inspired a series of works proving flexibility of different classes of varieties; see~\cite{A-8,AFKKZ-1, AKZ-1, APS, BKK, BG, BGS, D, FKZ, G, GSh, K, KM, MPS, PW, P1, P2, Sh-1}.

A popular class of varieties is formed by the closures of orbits of algebraic group actions. 
Note that it is a complicated question whether the closure of a given orbit is normal. Thus, it is natural to study non-normal orbits as well.
Some automorphisms of such a variety~$X$ are given by the group action. In general, these automorphisms are not sufficient to prove flexibility of $X$. However, often the variety $X$ has some extra automorphisms that allow one to prove flexibility. In this work we consider affine {\it complexity-zero horospherical varieties}, that is, affine varieties with an action of a connected linear algebraic group $G$ such that there is an open $G$-orbit $\mathcal{O}$ and the stabilizer of a point $x\in \mathcal{O}$ contains a maximal unipotent subgroup $U \subseteq G$. This class includes all toric varieties and is contained in the class of spherical varieties. A classification of complexity-zero horospherical affine varieties is obtained in~\cite{PV}. It is proved there that such varieties are in bijection with finitely generated subsemigroups of the semigroup of dominant weights. Many properties of a variety can be described in terms of the corresponding semigroup. We gather some preliminary information on horospherical varieties in Section~\ref{subsecthorospher}.

There is a series of works devoted to flexibility of certain subclasses in the class of horospherical varieties. The first step
is done in~\cite{AKZ-1}, where flexibility of non-degenerate normal toric varieties and cones over flag varieties is proved. Subsequently, in~\cite{Sh-1}, flexibility of (possibly non-normal) complexity-zero horospherical varieties corresponding to a semisimple group~$G$ is proved. In~\cite{GSh}, it is proved that every normal complexity-zero horospherical variety with only constant invertible functions is flexible in the case of an arbitrary algebraic group~$G$. In~\cite{BG}, a criterion for a not necessarily normal toric variety to be flexible is obtained. It is also worth mentioning that in~\cite{BGS} all orbits of the group $\mathrm{AAut}(X)$ on an arbitrary complexity-zero horospherical variety are described, where $\mathrm{AAut}(X)$ is the group generated by all connected algebraic subgroups of the automorphism group $\Aut(X)$. However, this description is in terms of degrees of homogeneous locally nilpotent derivations on~$\KK[X]$. In the non-toric case it is an open problem to find all these degrees. Therefore, this result does not yield a criterion for flexibility. 

It is easy to see that for any horospherical variety one can assume that $G=G'\times L$, where~$G'$ is a semisimple group and $L\cong (\KK^\times)^m$ is an algebraic torus. The results above give a criterion for a complexity-zero horosperical variety $X$ to be flexible if one of the factors~$G'$ or $L$ is trivial and for arbitrary $G$ in case when $X$ is normal. So the remaining case is that of a non-normal complexity-zero horospherical variety for which both factors are nontrivial. In the present work we fill this gap and obtain a criterion of flexibility in the general case, see Theorem~\ref{thmcrit}. All the previous results mentioned above can be obtained as particular cases of this criterion. The statement and the plan of the proof of this result are rather similar to the criterion of flexibility of toric varieties given in~\cite{BG}. However, the technique is more involved. So, the proof is based on new ideas and combining  techniques from all the papers mentioned above. 

A variety $X$ is called {\it generically flexible} if it possesses at least one flexible point. This condition is equivalent to $X$ admitting an open orbit consisting of flexible points. This is also equivalent to triviality of field Makar-Limanov invariant. It is known that generic flexibility of $X$ is not equivalent to flexibility. Examples of generically flexible but not flexible varieties are given in~\cite{Ko}; see also~\cite{G2}. However, in~\cite{GSh}, a technique for moving a regular point from an orbit to the open one is developed for complexity-zero horospherical varieties. This shows that for complexity-zero horospherical varieties flexibility and generic flexibility are equivalent. So, the main part of this work is to describe 
generically flexible complexity-zero horospherical varieties.

The authors are grateful to Ivan Arzhantsev and Roman Avdeev for useful discussions.

\section{Preliminaries}
\subsection{Locally nilpotent derivations} 
In this section we recall some necessary definitions related to locally nilpotent derivations and the correspondence between locally nilpotent derivations of an algebra and $\Ga$-actions; see~\cite{Fr} for detailed information and proofs.

Let $B$ be a $\KK$-domain. A \textit{derivation} of~$B$ is a linear map $\delta: B \to B$ satisfying the Leibniz rule $\delta(ab)=a\delta(b)+b\delta(a)$. A derivation $\delta$ of $B$ is called \textit{locally nilpotent} (LND) if for every $b \in B$ there exists a positive integer $m$ such that~$\delta^m(b) = 0$. The set of all locally nilpotent derivations of $B$ is denoted by $\LND(B)$.

Given an LND $\delta$ of $B$, one can construct an algebraic action of the group $\Ga$ on $B$ via the exponential map: $$s\cdot b=\mathrm{exp}(s\delta)(b)=\sum_{i=0}^{\infty}\frac{s^i}{i!}\delta^i(b)$$ 
for all $s \in \KK$ and $b \in B$. Conversely, every algebraic $\Ga$-action on $B$ determines an LND $\delta$ of $B$ given by 
$$\delta(b)=\frac{s\cdot b-b}{s}\Big|_{s=0}.$$ This correspondence is bijective. Moreover, the kernel of an LND coincides with the algebra of invariants of the corresponding $\Ga$-action. When $B$ is finitely generated, it is isomorphic to the coordinate ring of the affine algebraic variety $\mathrm{Spec}(B)$. Since every automorphism~$\phi$ of the variety $\mathrm{Spec}(B)$ induces an automorphism $\phi^*$ of the algebra $B$, and this correspondence is bijective, there exists a one-to-one correspondence between LNDs on $B$ and algebraic $\Ga$-actions on $\mathrm{Spec}(B)$.

Let $n$ be the transcendence degree of $B$ over $\KK$. The transcendence degree of the kernel of an LND $\delta$ equals $n-1$. In particular, if $A\subseteq \Ker \delta$ is an algebraically closed subalgebra with $\mathrm{tr.deg}_{\KK}A=n-1$, then $\Ker\delta=A$. 

Suppose $B$ is equipped with a grading by an abelian group $F$: 
$$B=\bigoplus_{f\in F}B_f,\text{ where } B_{f_1}B_{f_2}\subseteq B_{f_1+f_2} \text{ for all }f_1,f_2\in F .$$ A derivation $\delta$ of $B$ is called \textit{homogeneous} if it takes homogeneous elements to homogeneous ones. It can be shown that for any homogeneous derivation $\delta$ there exists an element $d \in F$ such that $\delta (B_f)\subseteq B_{f+d}$ for all $f\in F$. If $\delta \neq 0$, then such an element $d$ is uniquely determined and is called the \textit{degree} of $\delta$.
If $B$ is finitely generated, each derivation can be decomposed into a sum of homogeneous ones, which are called~{\it homogeneous components of the derivation}. 

Suppose we have a $\ZZ$-grading on a finitely generated $\KK$-domain $B$, and let $\delta$ be an LND on $B$. Then $\delta=\sum_{i=l}^k\delta_i$, where $\delta_i$ is a homogeneous derivation of degree $i$. We assume $\delta_l\neq 0$ and $\delta_k\neq 0$. Then, by~\cite{Re}, the derivations $\delta_l$ and $\delta_k$ are locally nilpotent. Applying this assertion $n$ times, we obtain the following lemma.

\begin{lemma}\label{pslemmar}
    Suppose we have a $\ZZ^n$-grading on a finitely generated $\KK$-domain $B$. Let $\delta$ be an LND of $B$. Consider the convex hull $P$ of degrees of nonzero homogeneous components of $\delta$. Then the homogeneous components of $\delta$ corresponding to the vertices of $P$ are locally nilpotent.
\end{lemma}

\subsection{Rational polyhedral cones}
Let $V$  be a finite-dimensional $\QQ$-vector space. A subset $\sigma\subseteq V$ is called a \textit{finitely generated cone} if there exists a finite set $v_1,\dots,v_k\in V$ such that $\sigma = \QQ_{\geq 0}v_1+\dots+\QQ_{\geq 0}v_k$. All cones considered in this paper are assumed to be finitely generated. 

Let $\sigma\subseteq V$ be a finitely generated cone. Denote by $\langle\cdot,\cdot\rangle$ the natural pairing between $V$ and the dual space $V^*$. Consider the \textit{dual cone} $\sigma^\vee$ in $V^*$, defined by $$\sigma^\vee:=\{ w\in V^* \mid \langle w,v \rangle\geq 0 \text{ for all } v\in \sigma\}.$$ 
Since the cone $\sigma$ is finitely generated, the cone $\sigma^\vee$ is also finitely generated and we have $(\sigma^\vee)^{\vee}=\sigma$. A \textit{face} $\tau$ of $\sigma$ is a subset of $\sigma$ given by $$\tau=\sigma \cap u^\perp =\{ v\in \sigma\mid \langle u,v \rangle =0 \}$$ for some $u\in \sigma^\vee$. Each face is also a finitely generated cone. 

The \textit{dimension} of $\sigma$ is the dimension of its linear span $\QQ \sigma$. A cone $\sigma\subseteq V$ is called \textit{strictly convex} if $\sigma\cap (-\sigma)=\{0\}$. A cone $\sigma$ is strictly convex if and only if the cone $\sigma^\vee$ has dimension equal to the dimension of~$V$. An \textit{extremal ray} of $\sigma$ is a face of dimension one. If $\sigma$ is strictly convex, its extremal rays are cones $\QQ_{\geq 0}v_i$ for a minimal set of generators $\{v_1,\ldots v_k\}$. 

There is a one-to-one correspondence between faces of $\sigma$ and faces of $\sigma^\vee$. A face $\tau\preceq \sigma$ correspons to the face 
$$\widehat{\tau}=\sigma^\vee\cap \tau^\bot=\{u\in\sigma^\vee\mid \langle v,u\rangle=0\ \forall v\in\tau\}\preceq\sigma^\vee.$$  
Note that $\dim \tau+\dim \widehat{\tau}=\dim V$. 

Let $N$ be a lattice and $N_\QQ=N\otimes_{\ZZ}\QQ$ be a rational vector space spanned by $N$. Consider the lattice~$M=\mathrm{Hom}_\ZZ(N,\ZZ)$ and the vector space $M_\QQ=M\otimes_{\ZZ}\QQ=N_\QQ^*$. Let $\sigma$ be a strictly convex cone in $N_\QQ$. Denote by $v_1,\ldots, v_k$ the primitive integer vectors on extremal rays $\rho_1,\ldots, \rho_k$ of~$\sigma$. For each ray $\rho_i=\QQ_{\geq 0}v_i$ we define $$\mathfrak{R}_\rho(\sigma) := \{ e\in M \mid \langle v_i,e \rangle =-1 \text{ and }  \langle v_j,e \rangle \geq 0 \text{ for all } j\neq i\}.$$ Elements of $\mathfrak{R}_\rho(\sigma)$ are called \textit{Demazure roots} of $\sigma$ {\it with distinguished ray} $\rho$.

By a semigroup we mean a finitely generated semigroup with unit. Consider a subsemigroup $\mathfrak{F}\subseteq M$ and the corresponding cone $\sigma^\vee:=\QQ_{\geq 0}\mathfrak{F}\subseteq M_\QQ$. We say that $\mathfrak{F}$ is \textit{saturated} if $\sigma^\vee\cap \ZZ \mathfrak{F}= \mathfrak{F}$. The semigroup $\mathfrak{F}_{sat}:=\sigma^\vee\cap \ZZ \mathfrak{F}$ is called the \textit{saturation} of~$\mathfrak{F}$ and elements of $\mathfrak{F}_{sat}\backslash \mathfrak{F}$ are called \textit{holes} of $\mathfrak{F}$. The following definitions are introduced in~\cite{TY}. An element $m$ of $\mathfrak{F}$ is said to be a \textit{saturation point} of $\mathfrak{F}$ if $(m+\sigma^\vee)\cap \ZZ \mathfrak{F}\subseteq \mathfrak{F}$, that is, the shifted cone $m+\sigma^\vee$ has no holes. A face $\tau$ of $\sigma^\vee$ is called \textit{almost saturated} if there exists a saturation point of $\mathfrak{F}$ in $\tau$. Otherwise $\tau$ is called \textit{nowhere saturated}.

\subsection{Horospherical varieties}
\label{subsecthorospher}

This section contains some preliminary information on horospherical varieties; see~\cite{PV, Ti-2} for proofs.

\begin{definition}
An irreducible algebraic variety $X$ is \textit{horospherical} if it admits an action of a connected linear algebraic group $G$ such that the stabilizer of a generic point contains a maximal unipotent subgroup of $G$. We say that a horospherical variety $X$ is \textit{complexity-zero} if the action of $G$ on $X$ possesses an open orbit $\mathcal{O}$. 
\end{definition}
If $X$ is an affine complexity-zero horospherical variety, then the unipotent radical of~$G$ acts trivially on $X$. Hence we may assume that $G$ is reductive. From now on by horospherical variety we mean a complexity-zero horospherical variety equipped with an action of a reductive group $G$. Moreover, coming up to a finite extension of the group we can assume that $G=G'\times L$, where $G'$ is a semisimple group and $L\cong (\KK^\times)^s$ is a torus.  

Let $B$ be a fixed Borel subgroup of $G$. Since $G=G'\times L$, we have $B=\widehat{B}\times L$, where $\widehat{B}$ is a Borel subgroup in $G'$. Let us denote by $\widehat{T}$ a maximal torus in $\widehat{B}$. Then $T=\widehat{T}\times L$ is a maximal torus in $B$.
For each character $\Lambda \in \mathfrak{X}(B)$ consider the weight subspace $$S_\Lambda := \{f \in \KK[G] \mid f(gb) = \Lambda(b)f(g) \text{ for all } g \in G, b \in  B \}.$$ The set $$\mathfrak{X}^{+}(B):=\{\Lambda \in \mathfrak{X}(B) \mid S_{\Lambda} \neq \{0\}\}$$ coincides with the set of dominant weights of $G$.

Let $\mathfrak{F} \subseteq \mathfrak{X}^+(B)$ be a (finitely generated) subsemigroup (with unit). Then the spectrum $X=X(\mathfrak{F})$ of the algebra $$S_{\mathfrak{F}}:=\bigoplus_{\Lambda\in \mathfrak{F}}S_{\Lambda}$$ is an affine horospherical $G$-variety, and all affine horospherical $G$-varieties can be obtained in such a way.

The affine horospherical variety $X=X(\mathfrak{F})$ can be obtained by the following construction. Let $\{\Lambda_1,\ldots, \Lambda_k\}$ be a system of generators of $\mathfrak{F}$. Let us consider the direct sum of irreducible representations of $G$ dual to representations with the highest weights $\Lambda_1,\ldots, \Lambda_k$:
$$
V=V(\Lambda_1)^*\oplus\ldots\oplus V(\Lambda_k)^*.
$$
Let $v_i$ be the highest vector of $V(\Lambda_i)^*$. Put $v=v_1+\ldots+v_k\in V$. Then $X$ is isomorphic to the orbit closure $\overline{Gv}$. 

Let $M:=\ZZ \mathfrak{F}$ be the lattice generated by $\mathfrak{F}$ and $N=\mathrm{Hom}(M,\ZZ)$ be its dual lattice. Consider the rational vector spaces $M_\QQ=M\otimes_{\ZZ}\QQ$ and $N_\QQ=N\otimes_{\ZZ}\QQ=M_\QQ^*$. Let $\sigma^{\vee}$ be the cone generated by~$\mathfrak{F}$, that is, 
$$\sigma^{\vee}:=\mathbb{Q}_{\geq 0} \mathfrak{F} \subseteq M_\QQ.$$  
By $\sigma$ we denote the cone in $N_{\QQ}$ dual to $\sigma^{\vee}$. 
Note that, by definition, the cone $\sigma^\vee$ is not contained in any hyperplane in $M_\QQ$. Hence $\sigma$ is strictly convex.

There is a one-to-one correspondence between faces of $\sigma$ and $G$-orbits on $X$. In particular, the number of orbits is finite. Namely, if $O_{\tau}\subseteq X$ is a $G$-orbit on $X$, corresponding to some face $\tau$ of $\sigma$, then the vanishing ideal of its closure $\overline{O_{\tau}}$ has the form $$I_\tau=I(O_{\tau})=\bigoplus_{\Lambda\in \mathfrak{F}\backslash \widehat{\tau}}S_{\Lambda}.$$
The subvariety $\overline{O_{\tau}}$ is a horospherical variety of the same group corresponding to the semigroup $\mathfrak{F}\cap \widehat{\tau}$. It is easy to see that $O_{\tau_1}\subseteq \overline{O_{\tau_2}}$ if and only if $\tau_2\preceq \tau_1$.

We need the following formula for the dimension of a horospherical variety given in~\cite[Section~3.2]{PV}. 

\begin{lemma}\label{lemorbit}
     Let $X$ be the horospherical variety corresponding to a subsemigroup $\mathfrak{F}\subseteq \mathfrak{X}^+(B)$. Denote by $\Delta(\mathfrak{F})$ the number of positive roots of $G'$ that can be obtained as sums of such simple roots that have nonzero inner product with at least one element of $\mathfrak{F}$. Then $\dim X=\dim\sigma^\vee+\Delta(\mathfrak{F})$.
\end{lemma}

This lemma implies the following describtion of orbits of codimension one. Let $\mathcal{C}$ be the fundamental Weyl chamber, that is the cone in $\mathfrak{X}(B)\otimes_\ZZ\QQ$ generated by $\mathfrak{X}^+(B)$. We denote by $\theta^\vee$ the cone $\mathcal{C}\cap M_\QQ$. Consider $\theta\subseteq N_\QQ$ dual to $\theta^\vee$.

\begin{lemma}\label{arl}
    Orbits of codimension one in $X$ correspond to extremal rays $\rho$ of $\sigma$ that are not extremal rays of $\theta$.
\end{lemma}
\begin{proof}
    Let $\tau$ be a face of $\sigma$ corresponding to an orbit of codimension one. Since 
$$\dim\widehat{\tau}+\Delta(\mathfrak{F}\cap\widehat{\tau})=\dim \overline{O_\tau}=\dim X-1=\dim\sigma^\vee+\Delta(\mathfrak{F})-1,$$ 
 $\dim \widehat{\tau}=\dim \sigma^\vee-\dim\tau\leq \dim \sigma^\vee-1$,  and $\Delta(\mathfrak{F}\cap\widehat{\tau})\leq \Delta(\mathfrak{F})$, we have $\dim\tau=1$ and $\Delta(\mathfrak{F}\cap\widehat{\tau})= \Delta(\mathfrak{F})$. The condition $\Delta(\mathfrak{F}\cap\widehat{\tau})= \Delta(\mathfrak{F})$ is equivalent to the fact that if $\widehat{\tau}$ is contained in a face of $\mathcal{C}$, then $\sigma^\vee$ is contained in this face. That is, $\widehat{\tau}$ is not contained in a proper face of $\theta^\vee$. That is $\tau$ is not an extremal ray of $\theta$.

Conversely, if $\rho$ is an extremal ray of $\sigma$ that is not an extremal ray of $\theta$, then $\Delta(\mathfrak{F}\cap\widehat{\rho})= \Delta(\mathfrak{F})$ and $\dim \overline{O_\rho}=\dim X-1$. 
\end{proof}

\begin{lemma}
\label{lemnormaliz}
The integer closure of $S_{\mathfrak{F}}$ in its quotient field coincides with $\bigoplus\limits_{\Lambda\in\mathfrak{F}_{sat}} S_\Lambda$.
\end{lemma}
\begin{proof}
    By the proof of~\cite[Proposition~6]{PV}, the integer closure of $S_{\mathfrak{F}}$ in its quotient field is contained in $S=\bigoplus\limits_{\Lambda\in \mathfrak{X}^{+}(B)} S_\Lambda$, since $S$ is factorial and hence normal. By~\cite[Proposition~4]{PV}, the integer closure of $S_{\mathfrak{F}}$ in $S$ is $\bigoplus\limits_{\Lambda\in\mathfrak{F}_{sat}} S_\Lambda$.
\end{proof}

\section{LNDs on affine horospherical varieties}

Let $X=X(\mathfrak{F})$ be an affine complexity-zero horospherical $G$-variety. Denote by $\sigma$ and $\sigma^{\vee}$ the corresponding pair of dual cones; see Section~\ref{subsecthorospher}. Let us define a new cone $\gamma=\gamma(X)$ corresponding to this variety. 

\begin{definition}
    An extremal ray $\rho\preceq \sigma$ is {\it significant} if $O_\rho$ has codimension one in $X$ and consists of regular points. 
Denote by $\gamma=\gamma(X)$ the cone in~$N_\QQ$ generated by~$\theta$ and 
the significant extremal rays of $\sigma$. We say that~$\gamma$ is the {\it regularity cone} of $X$. 
\end{definition}

\begin{lemma}
\label{lemKXreg}
The algebra of regular functions on the regular locus 
 $\Xreg$ of $X$ has the form $$\KK[\Xreg]=\bigoplus_{\Lambda\in \gamma^{\vee}\cap M}S_{\Lambda}.$$
\end{lemma} 
\begin{proof}
Since the open orbit $\mathcal{O}$ consists of regular points, we have the following chain of inclusions: $$\KK[X]\subseteq \KK[\Xreg]\subseteq \KK[\mathcal{O}].$$ 
By~\cite[Theorem~7]{PV}, we obtain
$$
\KK[\mathcal{O}]=\bigoplus\limits_{\Lambda\in \mathfrak{X}^+(B)\cap M} S_{\Lambda}=
\bigoplus\limits_{\theta^\vee\cap M}S_{\Lambda}.
$$
Since $\Xreg$ is a $T$-invariant subset of $X$, the algebra $\KK[\Xreg]$ is an $M$-graded subalgebra of~$\KK[\mathcal{O}]$. So, we need to determine which homogeneous elements of $\KK[\mathcal{O}]$ belong to $\KK[\Xreg]$. Let $f\in S_\Lambda\subseteq \KK[\mathcal{O}]$. Then $f\in \KK[\Xreg]$ if and only if $f$ is regular on each orbit of codimension one consisting of regular points, that is, $f$ is regular on $O_\rho$ for each significant extremal ray $\rho$ of $\sigma$. 

Recall that the ideal $I_\rho\subseteq \KK[X]$ of functions vanishing on the closure $D_\rho=\overline{O_\rho}$ has the form
$$
I_\rho=\bigoplus_{\Lambda\in M\cap\sigma^\vee,\  \langle \Lambda, v_\rho\rangle>0}\!\!\!\!\!\!S_\Lambda,\text{ where }v_\rho\text{ is the primitive vector on }\rho.
$$
If $g\in S_\Lambda \backslash \{0\}$ for some $\Lambda \in M\cap \sigma^\vee$, then the order $\nu_{D_\rho}(g)$ of $g$ along $D_\rho$ is positive if and only if $\langle \Lambda, v_\rho\rangle>0$. This easily implies that $\nu_{D_\rho}(g)=\langle \Lambda, v_\rho\rangle$. Note that for all 
$\Lambda\in M\cap\theta^\vee$ every element $f\in S_\Lambda$  equals $\frac{g}{h}$ for some $g\in S_{\Lambda_1}$, $h\in S_{\Lambda_2}$ and 
$\Lambda_1,\Lambda_2\in \sigma^\vee\cap M$.
Therefore, we have $\nu_{D_\rho}(f)=\langle \Lambda, v_\rho\rangle$. 

Thus, for $f\in S_\Lambda$ the conditions $\nu_{D_\rho}(f)\geq 0$ for all significant extremal rays $\rho$ are equivalent to the condition $\Lambda\in\gamma^\vee$.  
\end{proof}

\begin{lemma}
\label{lemregorb}
Let $\rho$ be an extremal ray  of the cone $\sigma$ such that $\mathrm{codim}_XO_{\rho}=1$. Then the following conditions are equivalent.
    \begin{enumerate}
        \item There exists an $M$-homogeneous LND $\delta_e$ on $\KK[X]$ of degree $e$, where $e$ is a Demazure root of $\sigma$ with the distinguished extremal ray $\rho$.
        \item The orbit $O_{\rho}$ consists of regular points of $X$, i.e. $\rho$ is significant.
        \item The face $\widehat{\rho}\preceq \sigma^\vee$ is almost saturated.
    \end{enumerate}
\end{lemma}
\begin{proof}
(1)$\implies$(2) Let us prove that there exists an automorphism ${\varphi\in\SAut(X)}$ such that~$\varphi$ takes a point in $O_\rho$ to a point in the open orbit $\mathcal{O}$. Since $\mathcal{O}$ consists of regular points, this implies that $O_\rho$ consists of regular points.

We claim that $\mathrm{Ker}\,\delta_e=\bigoplus\limits_{\Lambda\in M\cap \widehat{\rho}}S_\Lambda$. Indeed, since $\langle e,v_\rho\rangle=-1$, for every $\Lambda\in \widehat{\rho}$ we have $\Lambda+e\notin\sigma^\vee$. Therefore, $\bigoplus\limits_{\Lambda\in M\cap \widehat{\rho}}S_\Lambda\subseteq \mathrm{Ker}\,\delta_e$. But $\bigoplus\limits_{\Lambda\in M\cap \widehat{\rho}}S_\Lambda\cong \KK[O_\rho]$ is an algebraically closed subalgebra of $\KK[X]$ with transcendence degree $\dim X-1$. Hence, $${\bigoplus\limits_{\Lambda\in M\cap \widehat{\rho}}S_\Lambda= \mathrm{Ker}\,\delta_e}.$$


Denote by $\mathcal{H}_e$ the $\Ga$-subgroup in $\Aut(X)$ corresponding to $\delta_e$. There exists a function $f\in I_\rho$ such that $\delta_e(f)\notin I_\rho$. Indeed, we can take any $M$-homogeneous function $f$ such that $\delta_e(f)\neq 0$ and $\delta_e^2(f)=0$.

Take $x\in O_{\rho}$ such that $x$ does not belong to any orbit closure~$\overline{O_\tau}$, where $\tau\neq \rho$ is an extremal ray of $\sigma$. Moreover, we can choose $x$ in such a way that $\delta_e(f)(x)\neq0$ for some $f\in I_\rho$. For $s\in \mathcal{H}_e \backslash \{0\}$ we have
$$
f((-s)\cdot x)=s\cdot f(x)=\mathrm{exp}(s\delta_e)(f)(x)=(f+s\delta_e(f))(x)=s\delta_e(f)(x)  \neq 0.$$
Hence, we obtain $(-s)\cdot x\notin \overline{O_\rho}$.

The orbit $\mathcal{H}_e x$ is irreducible and it is not contained in $\overline{O_\zeta}$ for all extremal rays $\zeta$ of $\sigma$. Therefore, there exists a point $y\in \mathcal{H}_e x\cap \mathcal{O}$.

(2)$\implies$(1) Suppose 
the orbit $O_\rho$ consists of regular points of $X$. Consider the following $\mathbb{Z}$-grading of $\KK[X]$:
$$\KK[X]=\bigoplus_{i\in\mathbb{Z}}\KK[X]_i=\bigoplus_{\langle \Lambda,v_\rho \rangle =i} S_\Lambda.$$ 
All negative homogeneous components of this grading are zero.
The set of fixed points for the corresponding $\KK^\times$-action on $X$ is $Z=\mathbb{V}(I_\rho)=\overline{O}_\rho$. 
 
Now we are going to apply~\cite[Proposition~3]{GSh}. There the variety $X$ is assumed to be normal, but it is not used in the proof. So the conclusion holds for non-normal varieties as well. By~\cite[Proposition~3]{GSh}, the 
tangent space $\mathrm{T}_zX$ at every point $z\in Z^{\mathrm{reg}}\cap X^{\mathrm{reg}}\supseteq O_\rho$ is spanned by the tangent space $\mathrm{T}_zZ$ and the tangent vectors to the orbits of all regular $\Ga$-actions on $X$. Hence, there is a $\Ga$-action such that $\overline{O_\rho}$ is not invariant. Let $\delta$ be the corresponding LND of $\KK[X]$. 
Consider the decomposition $\delta=\sum_{\alpha\in M}\delta_\alpha$ of $\delta$ into a sum of $M$-homogeneous derivations. Since the ideal $I_\rho$ is not $\delta$-invariant, there exists $\alpha$ such that $\delta_\alpha\neq 0$ and $\langle \alpha,v_\rho \rangle <0$. Therefore, there exists a vertex $\beta$ of the convex hull of the degrees $\alpha$ satisfying $\delta_\alpha\neq 0$ such that $\langle \beta,v_\rho \rangle <0$. The derivation $\delta_\beta$ is locally nilpotent by Lemma~\ref{pslemmar}.
By~\cite[Lemma~4.3]{BGS}, the condition $\beta\notin\sigma$ implies that $\beta$ is a Demazure root of~$\sigma$. Since $\langle \beta,v_\rho\rangle<0$, the distinguished ray of $\delta_\beta$ is $\rho$.

(1)$\implies$(3)
  The maximal face, that is, the whole cone $\sigma^\vee$, is almost saturated; see for example~\cite[Lemma~2]{BG}. Suppose $p$ is a saturation point of $\sigma^\vee$. We denote $k=\langle p, v_\rho \rangle$. Let us prove that the point $q=p+ke\in\widehat{\rho}$ is a saturation point of $\widehat{\rho}$. Recall that $$\mathrm{Ker}\,\delta_e=\bigoplus\limits_{\Lambda\in M\cap\widehat{\rho}}S_\Lambda.$$ 
  Since $p$ is a saturation point of $\sigma^{\vee}$, for each $\alpha\in M\cap\sigma^\vee$ we have $p+\alpha\in\mathfrak{F}$. Let us take $f\in S_{p+\alpha}$. Then $\delta_e^k(f)\in S_{p+\alpha+ke}=S_{q+\alpha}$ is nonzero. Therefore, $q+\alpha\in\mathfrak{F}$. This implies that $q$ is a saturation point of $\widehat{\rho}$.

(3)$\implies$(1) Let us consider the normalization $\widetilde{X}$ of $X$. Then $\widetilde{X}$ is also a complexity-zero horospherical $G$-variety, and $\KK[\widetilde{X}]=\bigoplus\limits_{\Lambda\in\mathfrak{F}_{sat}} S_\Lambda$, where  $\mathfrak{F}_{sat}=M\cap\sigma^\vee$; see Lemma~\ref{lemnormaliz}. By definition, the cone generated by $\mathfrak{F}$ equals $\sigma^\vee$. Let us consider the orbit $O_\rho'$ on the variety~$\widetilde{X}$ corresponding to the ray $\rho\preceq\sigma$. By Lemma~\ref{lemorbit}, $\dim O_\rho'=\dim O_\rho=\dim X-1$. Since $\widetilde{X}$ is normal, the orbit $O_\rho'$ consists of regular points. But we have already proved the implication (2)$\implies$(1). So, there exists a Demazure root $e'$ of $\sigma^\vee$ with the distinguished ray $\rho$ and an $M$-homogeneous LND $\delta_{e'}$ of $\KK[\widetilde{X}]$ with degree $e'$. Since the face $\widehat{\rho}$ is almost saturated with respect to $\mathfrak{F}$, we have an $\mathfrak{F}$-saturation point $p\in\widehat{\rho}$. Let us fix a nonzero function $h\in S_p$ and consider a replica $h\delta_{e'}$ of $\delta_{e'}$. This replica $h\delta_{e'}$ is  an LND of $\KK[\widetilde{X}]$ such that $\KK[X]\subseteq \KK[\widetilde{X}]$ is an invariant subalgebra. Indeed, if $\Lambda\in\mathfrak{F}\setminus \widehat{\rho}$, then $\Lambda+p+e'=p+(\Lambda+e')\in \mathfrak{F}$. Hence, $\delta_e=h\delta_{e'}|_{\KK[X]}$ is an $M$-homogeneous LND of $\KK[X]$ with the degree $e=e'+p$, which is a Demazure root of $\sigma$ with the distinguished ray $\rho$.
\end{proof}

\begin{lemma}
\label{lemderu}
Under the conditions of Lemma~\ref{lemregorb}, the Demazure root $e$ can be chosen in~$\theta^\vee$. 
\end{lemma}
\begin{proof}
   If the intersection of $\theta^\vee$ and $\{m\in M\mid \langle m, v_\rho\rangle=-1 \}$ is empty, then $\widehat{\rho}$ is included in the boundary of $\theta^\vee$. This contradicts $\mathrm{codim}_XO_\rho=1$. Hence, we can take $m\in \theta^\vee$ such that $\langle m, v_\rho\rangle=-1$. Also, by Lemma~\ref{lemregorb}, there exists a Demazure root $e'\in M$ such that $\langle e', v_\rho\rangle=-1$, $\langle e', v_\zeta\rangle\geq 0$ for all extremal rays $\zeta\neq \rho$, and there exists a well-defined LND~$\delta_{e'}$ of degree $e'$ on $\KK[X]$. We can choose such $a\in \mathfrak{F}$ that $\langle a, v_\rho\rangle=0$, and $a+e'-m\in \sigma^\vee$. Let us denote $e=a+e'$. Then $\langle e, v_\rho\rangle=-1$, $\langle e, v_\zeta\rangle\geq 0$ for all extremal rays $\zeta\neq \rho$. Since $a=e-e'\in\mathfrak{F}$, we have $S_a\in\KK[X]$. Let us take a nonzero $f\subseteq S_a$. Since $\langle a,v_\rho\rangle=0$, we obtain $S_a\in\Ker \delta_{e'}$, in particular, $f\in\Ker \delta_{e'}$. Therefore,  there exists an LND $\delta_e=f\delta_{e'}$ of~$\KK[X]$. Also we have $e=(a+e'-m)+m\in \theta^\vee$.
\end{proof}

Let $e$ be a Demazure root of $\sigma$ such that there exists a well-defined LND $\delta_e$ of degree $e$. If $e$ is contained in $\theta^\vee$, we can provide an explicit construction of an LND with degree~$e$; compare~\cite[Lemma~5.8]{BGS}; see also~\cite[Section~6.1]{AA}. 
Let us fix $h\in S_e$. Then we define a linear mapping $D_h\colon \KK[X]\rightarrow \KK[X]$. Namely, for each homogeneous element $f\in S_\Lambda$ we put
$$D_h(f)=\langle \Lambda,v_\rho \rangle fh,$$
where $v_\rho$ is the primitive integer vector on the distinguished extremal ray $\rho$.
 
\begin{lemma}
\label{lemnext}
The mapping $D_h$ is a well-defined LND of degree $e$. 
\end{lemma}
\begin{proof}
    Since there exists a well-defined LND of degree $e$, for each $\Lambda\in \mathfrak{F}$ with $\langle \Lambda,v_\rho\rangle\neq 0$ we have $\Lambda+e\in \mathfrak{F}$. So, for $f\in S_\Lambda$, we obtain $fh\in\KK[X]$. If $\langle \Lambda,v_\rho\rangle= 0$, then $D_h(f)=0$. 
    Therefore, $D_h$ is well-defined. Let us check the Leibniz rule for $D_h$. It is sufficient to check it only for homogeneous elements. Let $f\in S_\Lambda$, $g\in S_\Omega$. Then
    $$
    D_h(fg)=\langle \Lambda+\Omega, v_\rho\rangle fgh=\langle \Lambda, v_\rho\rangle (fh)g+\langle \Omega, v_\rho\rangle (gh)f=D_h(f)g+D_h(g)f.
    $$
    Since $D_h^{\langle\Lambda,v_\rho\rangle+1}(f)=0$, the derivation $D_h$ is locally nilpotent. 
\end{proof}

\section{Main results}

The following flexibility criterion for an affine complexity-zero horospherical variety is the main result of this work.

\begin{theorem}
\label{thmcrit}
Let $X$ be an affine complexity-zero horospherical variety. Then the following conditions are equivalent. 
\begin{enumerate}
    \item The variety $X$ is flexible.
    \item The regularity cone $\gamma$ of the variety $X$ is not contained in any hyperplane in  $N_\mathbb{Q}$.
    \item For each non-constant regular function $f\in\KK[X]$ the intersection of the zero locus~$\mathbb{V}(f)$ and the regular locus $X^\mathrm{reg}$ is not empty.
\end{enumerate}
\end{theorem}
\begin{proof}
Firstly we prove that the conditions (2) and (3) are equivalent. Indeed, for a regular function $f$ the condition $\mathbb{V}(f)\subseteq X^{\mathrm{sing}}$ is equivalent to $f\in \KK[\Xreg]^\times$. Since $\KK[\Xreg]$ is an $M$-homogeneous subalgebra of $\KK[X]$, the regular locus $\Xreg$ admits a nonconstant invertible regular function if and only if it admits a homogeneous one. By Lemma~\ref{lemKXreg}, a homogeneous invertible function on $X^\mathrm{reg}$ exists if and only if $\gamma^\vee$ contains a nonzero subspace. This is equivalent to (2).

We proceed by proving the implication (1)$\implies$(2). By~\cite[Lemma~4]{BG} flexibility of $X$ implies $\KK[\Xreg]^\times=\KK^\times$. Therefore, $\Xreg$ does not admit non-constant invertible regular functions. Hence, we obtain (2).

Now let us prove the implication (2)$\implies$(1). First of all we are going to prove that $X$ has a flexible point. 
As before we assume that $G=G'\times L$, where $G'$ is semisimple and $L\cong (\KK^\times)^s$. So, for every point $z\in\mathcal{O}$ the tangent space  $\mathrm{T}_zX$ is the sum of tangent spaces to the orbits $G'z$ and $Lz$: $$\mathrm{T}_zX=\mathrm{T}_zG'z+\mathrm{T}_zLz.$$ 
Let $W$ be the subspace of $\mathrm{T}_zX$ generated by tangent vectors to orbits of $\Ga$-actions. Since~$G'$ is semisimple, it is generated by $\Ga$-subgroups. Therefore, $\mathrm{T}_zG'z\subseteq W$. Now we need to show that $\mathrm{T}_zLz\subseteq W$.

Let $\rho_1,\ldots, \rho_k$ be all extremal rays of the regularity cone $\gamma=\gamma(X)$ and $v_1,\ldots, v_k\in N$ be primitive vectors on these rays. Recall that by $\widehat{T}$ we denote a maximal torus of the Borel subgroup $\widehat{B}\subseteq G'$. We denote 
$$\mathfrak{X}(L)\otimes_\ZZ\QQ\cap M_\QQ=U.$$ 

Each $v_i$ is a linear function $M\rightarrow \ZZ$. Let $\alpha_i$ be the corresponding linear function $M_\QQ\rightarrow \QQ$ and $\overline{\alpha}_i\colon U\rightarrow \QQ$ be the restriction of $\alpha_i$ to  $U$.
If $\rho_i$ is an extremal ray of~$\theta$, then $\overline{\alpha}_i$ is the zero function.

Assume that the dimension of  the linear span of functions $\overline{\alpha_1},\dots,\overline{\alpha_k}$ is less than $\dim U$. Then $\mathrm{dim}(\mathrm{Ker}\,\overline{\alpha_1}\cap \dots \cap \mathrm{Ker}\,\overline{\alpha_k})\geq 1$. On the other hand,  
$$
\mathrm{Ker}\,\overline{\alpha_1}\cap \dots \cap \mathrm{Ker}\,\overline{\alpha_k}\subseteq \gamma^\vee.
$$
This contradicts $(2)$. So, the dimension of  the linear span of functions $\overline{\alpha_1},\dots,\overline{\alpha_k}$ equals $\dim U$.

We can assume that $\rho_{d+1},\ldots, \rho_k$ are all extremal rays of $\theta$ among $\rho_{1},\ldots, \rho_k$. Then  $\rho_1,\ldots, \rho_d$ are extremal rays of $\sigma$ corresponding to $G$-orbits of codimension one consisting of regular points. By Lemma~\ref{lemregorb}, for each $1\leq i\leq d$ there exists a Demazure root $e_i$ with the distinguished ray $\rho_i$ and an LND $\delta_{e_i}$ of degree $e_i$. By Lemma~\ref{lemderu},  we can take $e_i\in\theta^\vee$. Then by Lemma~\ref{lemnext} we can fix $h_i\in S_{e_j}\setminus\{0\}$ and consider the LND $D_{h_i}$ of $\KK[X]$. This LND corresponds to a $\Ga$-subgroup, which we denote $\mathcal{H}_{h_i}$.

Fix homogeneous generators  $f_1,\dots,f_q$ of $\KK[X]$, let $f_j\in S_{\Lambda_j}$. This set of generators defines an embedding $X\hookrightarrow \mathbb{A}^q$. The tangent vector to the $\mathcal{H}_{h_i}$-orbit at a point $z$ has the form
\begin{equation}
    \label{eqHeitangent}
    (\langle\Lambda_1,v_i\rangle f_1(z)h_i(z),\dots, \langle\Lambda_q,v_i\rangle f_q(z)h_i(z))=h_i(z)({\alpha_i}(\Lambda_1)f_1(z),\dots, {\alpha_i}(\Lambda_q)f_q(z)).
\end{equation}

The tangent space $\mathrm{T}_zLz$ to the $L$-orbit at $z$ is generated by tangent vectors to one-parameter subgroups of $L$. Let $\omega(t)\subseteq T=L\times \widehat{T}$ be a one-parameter subgroup. Then $\omega$ corresponds to a vector $\beta$ in $\overline{N}=\mathrm{Hom}(\mathfrak{X}(T),\ZZ)$ and the tangent vector to $\omega(t)z$ at unit equals 
\begin{equation}
\label{eqbetatangent}
    (\beta(\Lambda_1)f_1(z)),\dots , \beta(\Lambda_q)f_q(z)).
\end{equation} 

Note that linear functions $\beta\colon\mathfrak{X}(T)=\mathfrak{X}(L)\oplus \mathfrak{X}(\widehat{T})\rightarrow \ZZ$, where $\beta$ corresponds to $\omega(t)\in\widehat{T}$, can be realized as compositions of the projection $\pi\colon\mathfrak{X}(T)\rightarrow \mathfrak{X}(\widehat{T})$ and a function $\overline{\beta}\colon \mathfrak{X}(\widehat{T})\rightarrow \ZZ$. Therefore, each $\beta|_{M_\QQ}$ equals zero on $U$. Since functions $\overline{\beta}$ form the whole lattice $\mathrm{Hom}(\mathfrak{X}(\widehat{T}),\ZZ)$, the dimension of the linear span of $\beta\mid_{M_\QQ}$ equals $\dim M_\QQ-\dim U$. Since the dimension of the  linear span of $\overline{\alpha_i}=\alpha_i|_U$ equals $\dim U$, we can conclude that the dimension of the linear span of $\alpha_1,\ldots, \alpha_d$ and all $\beta\mid_{M_\QQ}$ corresponding to $\omega(t)\in\widehat{T}$, equals $\dim M_\QQ$. 
Therefore, for each $\xi\in\overline{N}$ corresponding to a one-parameter subgroup $\omega(s)\in L$, we have that $\xi|_{M_\QQ}$ is a linear combination of $\overline{\alpha_i}$, $1\leq i\leq d$ and $\beta_j|_{M_\QQ}$. 

Let us take a point $z\in\mathcal{O}$ such that all $f_j(z)$ and all $h_i(z)$ are nonzero. Then the tangent vector to any orbit of one-parameter subgroup in $L$ at $z$ has the form 
\begin{equation*}
    (\xi(\Lambda_1)f_1(z)),\dots , \xi(\Lambda_q)f_q(z)).
\end{equation*} 
So, it is contained in the linear span of tangent vectors to $\Ga$-subgroups of the form (\ref{eqHeitangent}) and one-parameter subgroups in $\widehat{T}$ of the form (\ref{eqbetatangent}). That is $\mathrm{T}_zLz\subseteq W$. So, the point $z$ is flexible, i.e. $X$ is generically flexible. Since there is a flexible point in $\mathcal{O}$, all points in $\mathcal{O}$ are flexible.

Now we show that each regular point of $X$ is flexible. We are going to prove this similarly to the proof of the implication (3)$\implies$(2) of Lemma~\ref{lemregorb}. Suppose a proper face $\widehat{\tau}$ of the cone~$\sigma^{\vee}$ corresponds to an orbit $O_\tau$ consisting of regular points. There exists such a primitive integer vector $v\in \sigma$ that $\widehat{\tau}=\sigma^\vee\cap v^\bot$. Consider the following $\ZZ$-grading on $\KK[X]$:
$$\KK[X]=\bigoplus_{i\in\mathbb{Z}}\KK[X]_i=\bigoplus_{\langle \Lambda,v \rangle =i} S_\Lambda.$$ 
Homogeneous components with negative $i$ are zero. Applying~\cite[Proposition~3]{GSh} we obtain that the closure of the orbit $O_\tau$ is not $\SAut(X)$-invariant. Hence, there exists an automorphism in $\SAut(X)$ that takes a point from $O_\tau$ to a point in an orbit of bigger dimension. Therefore, there exists an automorphism in $\SAut(X)$, taking a point from $O_\tau$ to a point in~$\mathcal{O}$. So, all points in $O_\tau$ are flexible. Thus, all regular points of $X$ are flexible. Theorem~\ref{thmcrit} is proved.
\end{proof}

Let us show that all the previous results about flexibility for subclasses of the class of horospherical varieties can be obtained as particular cases of Theorem~\ref{thmcrit}.

\begin{itemize}
    \item In the case of a semisimple group $G$, we have $\mathcal{C}$ is a strictly convex cone. Therefore, the cone $\theta^\vee$ is strictly convex. So, $\theta$ is not contained in any hyperplane in $N_\QQ$. Since $\gamma\supseteq \theta$, we conclude that  $\gamma$ is not contained in any hyperplane in $N_\QQ$. Therefore, by Theorem~\ref{thmcrit}, $X$ is flexible. This result coincides with \cite[Theorem~2]{Sh-1}. A particular case of this result is~\cite[Theorem~0.2(1)]{AKZ-1}. 
    \item In the case of $G=L$ is a torus, we have that all extremal rays of $\sigma$ correspond to orbits of codimension one. Also assuming that $G$ acts effectively in this case we have $\mathcal{C}=\theta^\vee=\mathfrak{X}(L)\otimes_{\ZZ}\QQ=M_\QQ$. Therefore, $\theta$ is a point and $\gamma$ is generated by extremal rays of $\sigma$ corresponding to orbits consisting of regular points. These are such extremal rays $\rho$ that $\widehat{\rho}$ is almost saturated. In this case Theorem~\ref{thmcrit} gives assertions of Theorem~1, Corollary~1, and Corollary~2 from~\cite{BG}. In the case of a normal variety we obtain~\cite[Theorem~0.2(2)]{AKZ-1}. 
    \item In the case of a normal variety $X$ all orbits of codimension one consist of regular points. That is all extremal rays of $\sigma$ are either extremal rays of $\theta$ or significant. So, $\gamma=\sigma$. Therefore, Theorem~\ref{thmcrit} states that $X$ is flexible if and only if $\sigma$ is not contained in any hyperplane of $N_\QQ$. This is equivalent to the fact that $X$ has only constant invertible regular functions. This result coincides with~\cite[Theorem~3]{GSh}.
\end{itemize}

Since a finite set of holes does not effect on almost saturation of a face, we can state the following corollary similar to~\cite[Corollary~3]{BG}.
\begin{cor}
    Suppose the semigroup $\mathfrak{F}$ has finitely many holes. Then the variety $X$ is flexible if and only if $\sigma$ is not contained in any hyperplane of $N_\QQ$. 
\end{cor}
\begin{proof}
    The cone $\gamma$ obtained from the semigroup $\mathfrak{F}$ coincides with cone $\gamma$ obtained from the semigroup $\mathfrak{F}_{sat}$ and coincides with $\sigma$.  
\end{proof}

This corollary can be reformulated as follows. 
\begin{cor}
    Suppose the semigroup $\mathfrak{F}$ has a finite number of holes. Then the variety $X$ is flexible if and only if it admits only constant invertible regular functions. 
\end{cor}

We finish the paper by considering two examples of horospherical $\mathrm{SL}_2(\KK)\times\KK^\times$-varieties. Let us consider the following $\mathrm{SL}_2(\KK)\times\KK^\times$-actions on $\KK[x,y,z]$:
$$
\left(\begin{pmatrix}
    a&b\\
    c&d
\end{pmatrix}, t\right)\cdot (x,y,z)=(ax+by,cx+dy, tz). 
$$

\begin{example}\label{eeeee}
    Let $X$ be the orbit closure $\overline{(\mathrm{SL}_2(\KK)\times\KK^\times)\cdot(x^2,xz, z)}$. Since 
    $$U=\left\{\begin{pmatrix}
        1&0\\
        c&1
    \end{pmatrix}\right\}$$ 
    is contained in the stabilizer of $(x^2,xz, z)$, this variety is a complexity-zero horospherical variety of the group $\mathrm{SL}_2(\KK)\times\KK^\times$. Let us consider the following basis in $\mathfrak{X}(B)$, where $\widehat{B}$ is the subgroup of upper triangular matrices: the first vector is the unique fundamental weight of $\widehat{B}$ and the second is the character $(A,t)\mapsto t$. Then $\mathcal{C}=\mathrm{cone}((1,0),(0,1),(0,-1))$ and~$\mathfrak{F}$ is generated by $(2,0), (1,1)$ and $(0,1)$. Therefore, 
    $$M=\ZZ^2,\qquad N=\ZZ^2,\qquad \theta^\vee=\mathcal{C},\qquad \theta=\mathrm{cone}((1,0)), $$ 
    $$\sigma^{\vee}=\mathrm{cone}((1,0),(0,1)), \qquad\sigma=\mathrm{cone}((1,0),(0,1)).$$ 
    It is easy to see that $\widehat{\rho_2}=\mathrm{cone}((1,0))$ is a nowhere saturated face of $\sigma^\vee$ and  $\widehat{\rho_1}=\mathrm{cone}((0,1))$ corresponds to the orbit $O_{\rho_1}$ of codimension two. Therefore, $\gamma=\theta=\mathrm{cone}((1,0))$. Hence, by Theorem~\ref{thmcrit}, the variety $X$ is not flexible.

\begin{figure}[h]
\begin{minipage}[h]{0.45\linewidth}
\center{\includegraphics[width=0.55\linewidth]{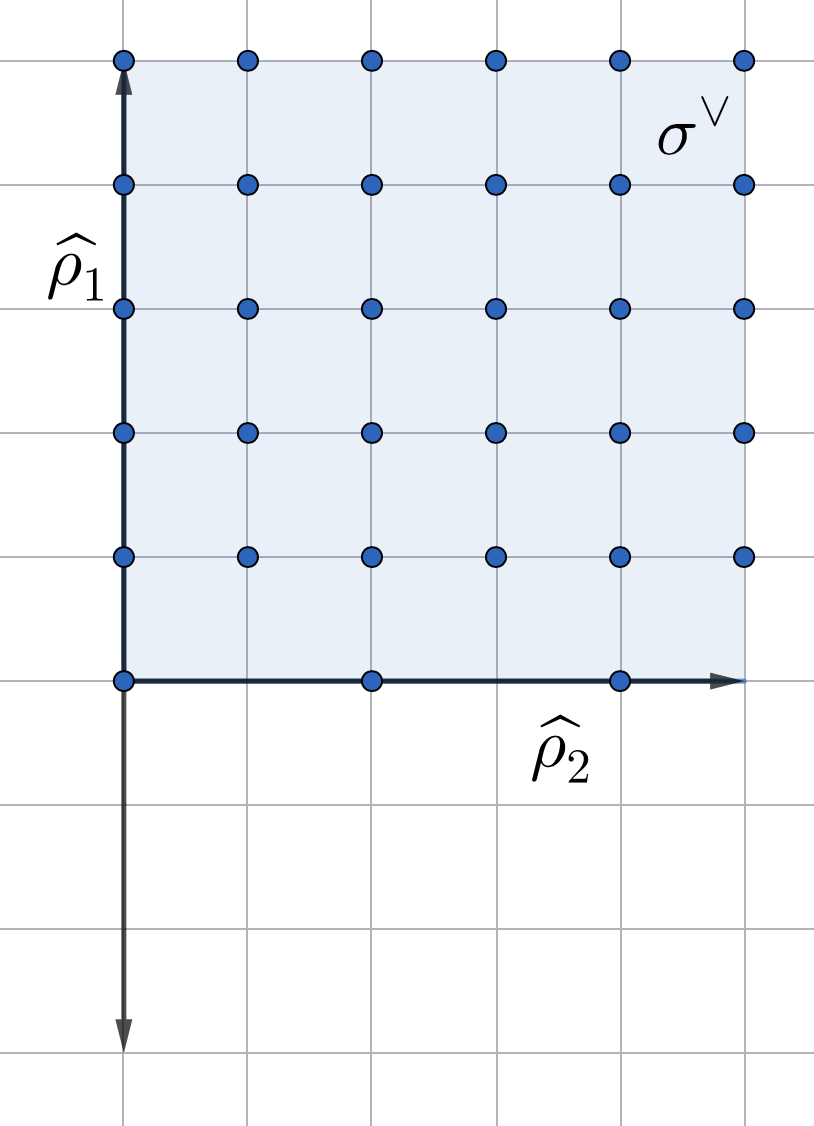}} \\
\end{minipage}
\hfill
\begin{minipage}[h]{0.45\linewidth}
\center{\includegraphics[width=0.55\linewidth]{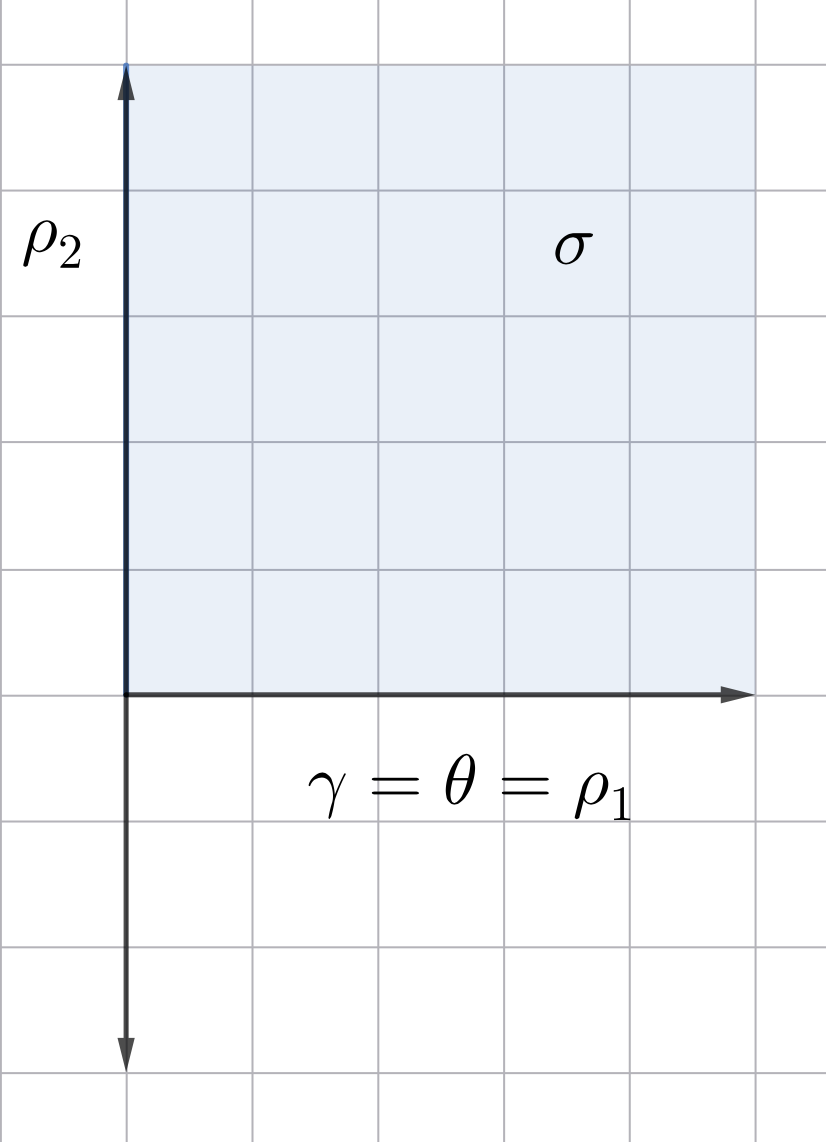}} 
\end{minipage}
\end{figure}
\end{example}

\begin{example}
Under the notation of Example~\ref{eeeee} let $Y=\overline{(\mathrm{SL}_2(\KK)\times\KK^\times)\cdot(x^2,xz, xz^2)}$. Then~$\mathfrak{F}$ is generated by $(2,0), (1,1)$ and $(1,2)$. Therefore, 
$$M=\ZZ^2,\qquad N=\ZZ^2,\qquad \theta^\vee=\mathcal{C}, \qquad\theta=\mathrm{cone}((1,0)),$$ 
$$\sigma^{\vee}=\mathrm{cone}((1,0),(1,2)),\qquad \sigma=\mathrm{cone}((2,-1),(0,1)).$$ 
It is easy to see that $\widehat{\rho_1}=\mathrm{cone}((1,2))$ is almost saturated and  $\widehat{\rho_2}=\mathrm{cone}((1,0))$ is nowhere saturated. Therefore, $\rho_1=\mathrm{cone}((2,-1))$ is a significant  extremal ray and $\rho_2=\mathrm{cone}((0,1))$ is not. Hence, $\gamma=\mathrm{cone}((1,0), (2,-1))$. By Theorem~\ref{thmcrit}, the variety $X$ is flexible.
\begin{figure}[h]
\begin{minipage}[h]{0.45\linewidth}
\center{\includegraphics[width=0.55\linewidth]{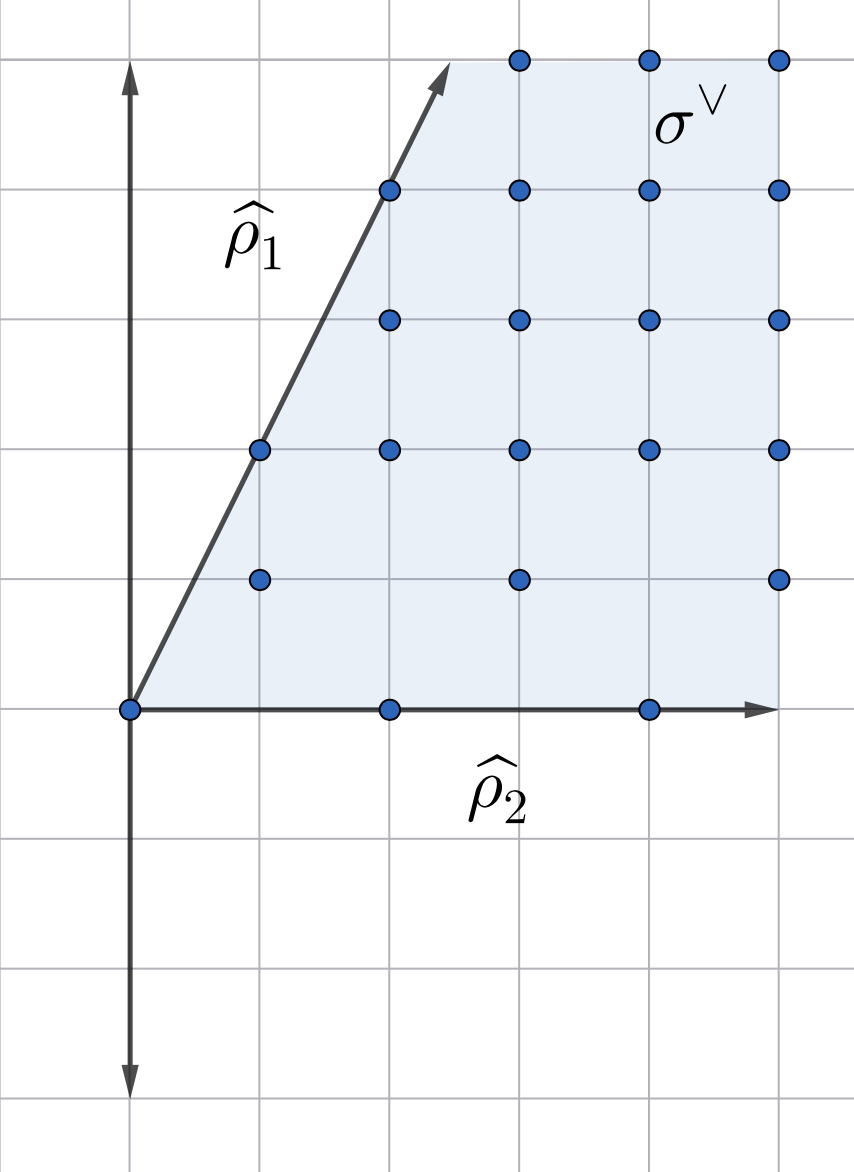}} \\
\end{minipage}
\hfill
\begin{minipage}[h]{0.45\linewidth}
\center{\includegraphics[width=0.55\linewidth]{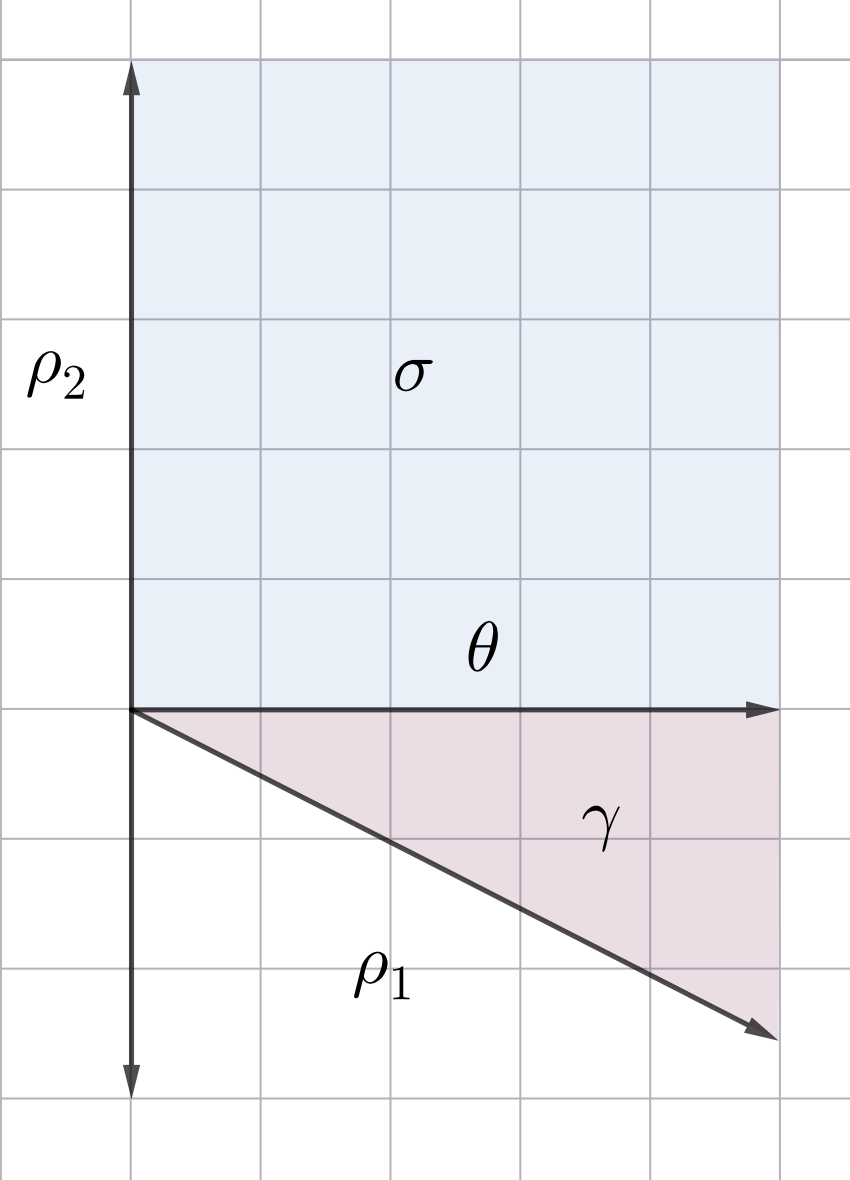}} 
\end{minipage}
\end{figure}
\end{example}

\end{document}